\documentclass[review]{elsarticle}
\usepackage{color}
\usepackage{lipsum} 
\usepackage{xcolor}
\usepackage{lineno}
\usepackage{amsmath,amsthm,amssymb}
\usepackage{amsfonts}
\usepackage{amscd}

\usepackage{graphicx}
\usepackage{enumitem}
\usepackage{mathtools}

\DeclareMathOperator*{\Int}{Int}
\usepackage{enumitem}

\usepackage{amsmath}
\usepackage{amsfonts}
\usepackage{amssymb}
\usepackage{amsthm}
\usepackage{amscd}
\usepackage[hidelinks]{hyperref}
\usepackage[ruled,linesnumbered]{algorithm2e}
\usepackage[numbers]{natbib}
\usepackage{mathrsfs}
\usepackage{color}
\usepackage{lipsum} 
\usepackage{xcolor}

\usepackage{calligra}
\newtheorem*{h1*}{Hypothesis H1}
\newtheorem*{h2*}{Hypothesis H2}
\newcommand{\h}{\text{\calligra h}\,}

\newcounter{cases}
\newcounter{subcases}[cases]

\modulolinenumbers[5]

\journal{Inverse Problems}

\begin{document}
\newcommand{\R}{\mathbb{R}}
\newcommand{\origin}{\mathcal{O}}
\newcommand{\sphere}{\mathbb{S}}

\newtheorem{Def}{Definition}[section]
\newtheorem{claim}{Claim}[section]
\newtheorem{lemma}{Lemma}[section]
\newtheorem{prop}{Proposition}[section]
\newtheorem{conj}{Conjecture}[section]
\newtheorem{theorem}{Theorem}[section]
\newtheorem{coro}{Corollary}[section]

\begin{frontmatter}

\title{Convex Solutions to the Virtual Source Reflector Problem}

\author{Dylanger S. Pittman}
\address{400 Dowman Drive, Atlanta}

\ead{dpittm2@emory.edu}

\begin{abstract}
We greatly expand upon the results of Kochengin, Oliker and Tempeski \cite{KOT} to include results for uniqueness in the general case. We also include results for existence in the rotationally symmetric case and the case where the target set is sufficiently small. We also point out an error that was found in \cite{KOT}.
\end{abstract}


\end{frontmatter}

\section{Introduction}

Let $\origin$ be the origin of $\R^3$, and let $\sphere^2$ be the unit sphere centered at $\origin$. We treat points on $\sphere^2$ as unit vectors with initial points at $\origin.$ Let an {\it aperture} be a subset of $\sphere^2$; in our work, the aperture will be an open set. Physically, it makes sense to consider $\origin$ as the location of an anisotropic point source of light such that rays of light are emitted in a set of directions defined by an aperture $D\subseteq \sphere^2$.

\begin{Def}\label{reflectorDef}
    Assume that we are given an aperture that is a connected open set $D\subseteq \sphere^2$, and a function $\rho:D\to (0,\infty)$ that is continuous and almost everywhere differentiable. Then a {\bf reflector} is the set $R=\{m\rho(m)|m\in D\}\subset \R^3$. 
\end{Def}

We first recall the classical law of reflection. Assume that we have a continuous, almost everywhere differentiable, positive function $\rho:D\to (0,\infty)$ and a corresponding reflector $R=\{m\rho(m)|m\in D\}.$ Suppose that a ray originating from $\origin$ in the direction $m\in D$ is incident on the reflector $R$ at the point $m\rho(m)$. If $\rho$ is differentiable at $m$, there is a unit vector, $n(m)$, normal to the reflector $R$ at $m\rho(m)$. Therefore, by the reflection law of geometric optics, a ray from $\origin$ of direction $m$ reflects off the point $m\rho(m)$ in the direction 
\begin{equation}\label{reflect}
    {y}(m)=m-2\langle m,n(m)\rangle n(m)
\end{equation}
where $\langle m,n(m)\rangle$ is the standard Euclidean inner product in $\R^3$ and $n(m)$ is oriented such that $\langle m,n(m)\rangle>0$ \cite{born_wolf_2019}. 

\begin{Def}\label{refractdef}
    Assume that we are given an aperture that is a connected open set $D\subseteq \sphere^2$.  Let $U$ be an open subset of $\sphere^2$ such that $D\subseteq U$. Consider a function $\rho:U\to (0,\infty)$ that is continuous and almost everywhere differentiable. Then a {\bf refractor} is the set $R=\{m\rho(m)|m\in U\}\subset \R^3$. 
\end{Def}

Note that $R=\{m\rho(m)|m\in D\}$ can be considered as either a reflector or a refractor. If $R=\{m\rho(m)|m\in D\}$ is considered as a refractor, the refracted direction $\hat{y}$ is determined by Snell's law and is given as 
\begin{equation}\label{refract}
    \hat{y}(m)=c_fm-\left(\sqrt{1-c_f^2(1-\langle m,n(m)\rangle^2)}-c_f\langle m,n(m)\rangle\right) n(m)
\end{equation}
where $c_f$ denotes the refraction index.

We borrow the following motivation from \cite{KOT}. Consider a two-sheeted hyperboloid of revolution with sheets $B$ and $H$. Let $\origin$ be the focus inside the convex body bounded by the first sheet $B$ and $x$ the focus inside the convex body bounded by the sheet $H$. Suppose that a point source of light is positioned at $\origin$ and the sheet $H$ is a reflector. $H$ has very special and important reflecting properties. Specifically, if a ray of direction $m$ from $\origin$ is incident on a point $z\in H$ and is reflected in the direction $y(m)$ as defined by (\ref{reflect}), then the ray from $z$ of direction $y(m)$ coincides with a ray from $x$ of direction $y(m)$. This means that the focus $x$ can be viewed, from a physical perspective, as a virtual source of rays reflected off $H$. A two-dimensional analog of this situation is illustrated in Figure \ref{figure_virtual-source}.

\begin{figure}[htbp]
\centering
\label{figure_virtual-source}
\def\svgwidth{0.9\textwidth}
\begingroup%
  \makeatletter%
  \providecommand\color[2][]{%
    \errmessage{(Inkscape) Color is used for the text in Inkscape, but the package 'color.sty' is not loaded}%
    \renewcommand\color[2][]{}%
  }%
  \providecommand\transparent[1]{%
    \errmessage{(Inkscape) Transparency is used (non-zero) for the text in Inkscape, but the package 'transparent.sty' is not loaded}%
    \renewcommand\transparent[1]{}%
  }%
  \providecommand\rotatebox[2]{#2}%
  \newcommand*\fsize{\dimexpr\f@size pt\relax}%
  \newcommand*\lineheight[1]{\fontsize{\fsize}{#1\fsize}\selectfont}%
  \ifx\svgwidth\undefined%
    \setlength{\unitlength}{595.27559055bp}%
    \ifx\svgscale\undefined%
      \relax%
    \else%
      \setlength{\unitlength}{\unitlength * \real{\svgscale}}%
    \fi%
  \else%
    \setlength{\unitlength}{\svgwidth}%
  \fi%
  \global\let\svgwidth\undefined%
  \global\let\svgscale\undefined%
  \makeatother%
  \begin{picture}(1,1.41428571)%
    \lineheight{1}%
    \setlength\tabcolsep{0pt}%
    \put(0,0){\includegraphics[width=\unitlength,page=1]{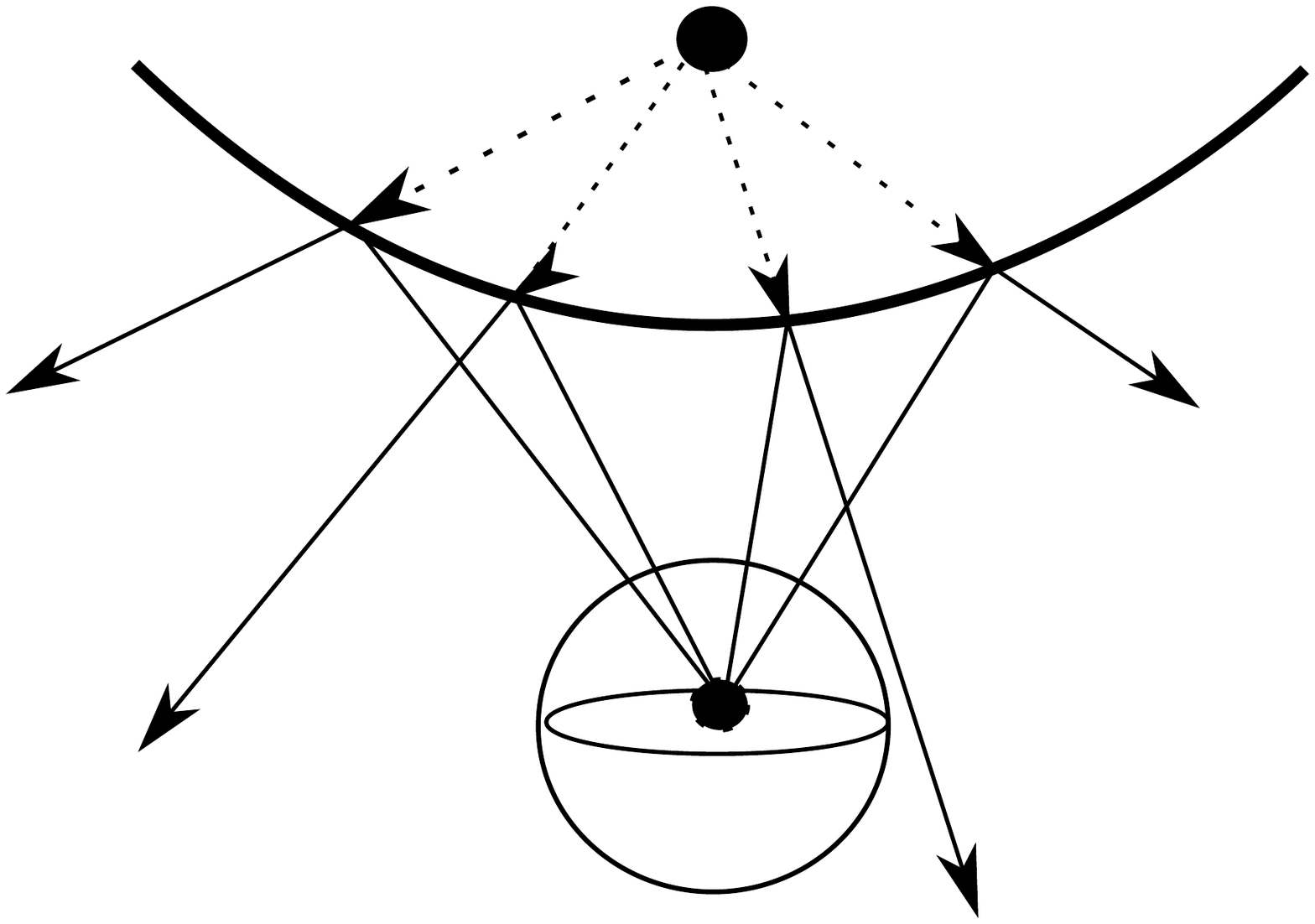}}%
    \put(0.0888283,0.65663263){\color[rgb]{0,0,0}\makebox(0,0)[lt]{\lineheight{1.25}\smash{\begin{tabular}[t]{l}$\mathbb{S}^2$ centered at the origin $\mathcal{O}$ \end{tabular}}}}%
    \put(0.57999653,0.90395853){\color[rgb]{0,0,0}\makebox(0,0)[lt]{\lineheight{1.25}\smash{\begin{tabular}[t]{l}light rays\end{tabular}}}}%
    \put(0.28738563,1.21921181){\color[rgb]{0,0,0}\makebox(0,0)[lt]{\lineheight{1.25}\smash{\begin{tabular}[t]{l}Target set consisting of a single point\end{tabular}}}}%
    \put(0.12366287,0.95969387){\color[rgb]{0,0,0}\makebox(0,0)[lt]{\lineheight{1.25}\smash{\begin{tabular}[t]{l}The hyperboloid $H$\end{tabular}}}}%
  \end{picture}%
\endgroup%

\caption[Virtual Source Reflector Illustration]{Here is an illustration of a virtual source reflector system where the target set is a single point. Note that all the rays of light reflect off of the hyperboloid $H$ such that it appears that the light is originating from the target point. }
\end{figure}

This same situation can also be interpreted from a different point of view allowing us to treat it geometrically as a refraction problem, rather than a reflection problem. Now suppose a light ray of direction $m$ from $\origin$ strikes $H$ and `refracts' such that the refracted direction is given by
\begin{equation}\label{reflect-refract}
    \hat{y}=-{y}=-m+2\langle m,n(m)\rangle n(m).
\end{equation}
Then since the refracted direction is the opposite of the reflected direction, every ray of direction $m$ that strikes the refractor $H$ will cross the focus $x$. Equation (\ref{reflect-refract}) can also be considered as the version equation of (\ref{refract}) where $c_f=-1.$ Under the law of total energy conservation, the total energy `delivered' by the refractor $H$ to the point $x$ will be equal to the total energy produced by the source $\origin$. We will only discuss this type of refraction, where $c_f=-1$, for the rest of the dissertation.

This interpretation of the reflection with a virtual source as a particular case of refraction is convenient from a geometric point of view and we use this terminology throughout this paper. Physically, however, it is more natural to treat the point $x$ as a virtual source. This would also be consistent with the case of a distributed virtual source; which we focus on. 

To quote \cite{KOT}: with this terminology, the problem studied in this paper can now be described as a problem of finding a convex refractor $R$ which will refract a given anisotropic bundle of rays from a source $\origin$ in such a way that the refracted rays are incident on a specified set in space and produce there, a given-in-advance intensity distribution. More specifically, suppose that we have a system consisting of an anisotropic point source at $\origin$, an aperture $D$, a nonnegative $g(m)\in L^1(\sphere^2)$, a target set $T\subset \R^3\setminus\{\origin\}$, and nonnegative integrable function $f$ defined on $T.$ The problem consists of finding a refractor $R$ which produces the specified in advance $f$ on $T.$ Henceforth, we call this problem the {\it refractor problem}.

The only previous work available with respect to this problem can be found in \cite{KOT}. In the paper, they develop a definition of a weak solution to a PDE of Monge Amp\`ere type; specifically, the PDE described by equation (4) in \cite{KOT}. They detail the construction of simply connected convex refractors and provide an existence theorem for the case where the target set is discrete (Theorem \ref{KOTthrm}). Due to the weak convergence of Dirac measures to Lebesgue measures, one can create refractors that produce discrete irradiance distributions that are arbitrarily close to a continuous distribution, like pixels in a photo. However, this does not imply the existence of a refractor that produces a continuous intensity distribution at the limit. This is expected for problems that can be described by a fully nonlinear PDE of Monge-Amp\'ere type \cite{Oliker1989}. However, if the refractors are convex, due to the unique properties that convexity provides, one can use the weak convergence of Dirac measures to Lebesgue measures to obtain a refractor that produces a continuous irradiance distribution; see \cite{OK1},\cite{para}. 
 
In this paper, we work on the weak formulation developed by \cite{KOT}, where we develop existence and uniqueness results. Due to a mistake in \cite{KOT} (see Section \ref{app:wrongKOT}), Theorem 9 in \cite{KOT} (that I later present as Theorem \ref{KOTthrm}) is the only existence theorem for the refractor problem. That theorem is hard to use, and, in its current form, it cannot be extended to the continuous case. However, Theorem 9 in \cite{KOT} can be used to prove another existence theorem for the discrete case (Theorem \ref{convRefr}) that, in turn, can be extended to the continuous case (Theorem \ref{continuous}). We use this result to then prove the existence of solutions for the rotationally symmetric case (Theorem \ref{cov-rot}). Additionally, we prove a uniqueness theorem for the case where the target set is finite (Theorem \ref{unique}) and for the general case (Theorem \ref{unique-continuous}).

\section{Hyperboloids of Revolution}\label{hyper-background}

We do all our work in $\R^3$. We denote $\sphere^2$ to be the unit sphere with the center at $\origin$ and $k_x=x/|x|$ for all $x\in \R^3\setminus\{\origin\}.$ We borrow much of this geometric setup from \cite{KOT}. Hyperboloids of revolution are of paramount importance when solving the virtual-source reflector problem due to their unique optical properties.

Consider the rotationally symmetric hyperboloid of two sheets in $\R^3$ such that one focus is $\origin$ and the other is $x;$ let $H(x)$ be the branch of the hyperboloid that has $x$ as a focus. From now on, when we use the term {\it hyperboloid}, we are only referring to this branch. 

With each hyperboloid $H(x)$ we associate its radial projection by rays from the origin onto an open spherical disk $D(x)\subset \sphere^2$ and its polar radius
\begin{equation}
    \h_\epsilon(m)=\frac{|x|(1-\epsilon^2)}{2\epsilon(1-\epsilon\langle m , k_{x}\rangle)}, \; m\in D(x)
\end{equation}
where $\epsilon$ is the eccentricity of the hyperboloid. Be aware that $\epsilon>1$ since we are describing a hyperboloid. 

Define $H_{\epsilon}(x)$ to be the hyperboloid with eccentricity $\epsilon$ and focus $x$. We now introduce a similar function $\h_{x,\epsilon}(m)$ which introduces $x\in \mathbb{R}^3\setminus \{\origin\}$ as a variable.  In this paper, we define $h_{x,\epsilon} (m)=m\h_{x,\epsilon}(m)$ for $m\in D(x)$ and $x\in \mathbb{R}^3\setminus \{\origin\}$. Let  $D_\epsilon(x)\subset \mathbb{S}^2$ be the preimage of $H_{\epsilon}(x)$ under $h_{x,\epsilon}$, then $D_\epsilon(x)=\{m\in \mathbb{S}^2|\frac{1}{\epsilon} <\langle m , k_{x}\rangle\}$. Thus we can easily verify that $H_\epsilon(x)=\{h_{x,\epsilon}(m)|m\in D_\epsilon(x)\}$.

From a physical perspective, $x$ being the focus means that all light from the origin reflected off of the reflector $H(x)$ appears to be originating from $x$, making $x$ a virtual source.

From the above work, we see by taking the eccentricity $\epsilon$ to infinity that the shape of the hyperboloid becomes a plane, which is the {\it directrix} of the hyperboloid, and $D_\epsilon(x)$ and goes to the hemisphere oriented towards $x$. The following two propositions summarize what I say precisely. 

\begin{prop}\label{eccdom}
 As the eccentricity $\epsilon$ of $H_{\epsilon}(x)$ goes to infinity, $D_\epsilon(x)$ goes to $\{m\in \sphere^2|\langle m,k_x\rangle\ge 0\}.$
\end{prop}

\begin{prop} \label{eccplane}
 As the eccentricity $\epsilon$ of $H_{\epsilon}(x)$ goes to infinity, the resultant set is a plane represented by the equation $\langle x,y-\frac{x}{2} \rangle=0$ where $y\in\R^3,$ or equivalently by the polar radius equation $r(m)=\frac{|x|}{2\langle m,k_x\rangle}$ where $m\in\{m\in \sphere^2|\langle m,k_x\rangle> 0\}.$
\end{prop}

Observe that as the eccentricity $\epsilon$ goes to $1$, we obtain a ray originating at $x$ going in the direction described by the vector $k_x$. We call this a {\it degenerate hyperboloid.}

An important property of hyperboloids can be described by the following proposition. 

\begin{prop}
Let $c>0$ and $\epsilon>1$ such that $c\epsilon>1$. Then the hyperboloids $H_{c\epsilon}(x)$ and $H_{\epsilon}(x)$ have the same foci: $\origin$ and $x$.
\end{prop}

The aforementioned property is important because a reflector $H_\epsilon(x)$ will reflect the light emitted from $\origin$ so that the light appears to be emitted from $x$; thus, making $x$ a virtual source. Alternatively, a refractor $H_\epsilon(x)$ will refract the light emitted from $\origin$ so that the light is delivered to $x$. These properties are true no matter how large or small the eccentricity is; all that matters is the location of the foci. 

Let $A_x=\{m\in \sphere^2| \langle m,k_x \rangle \ge 0\}$ and $A_x^\delta =\{m\in \sphere^2| \langle m,k_x \rangle \ge \delta\}$ for $\delta \in \R.$ Observe that $A_x=A_x^0$, $A_x^1=\{k_x\}$, $A_x^\delta= \varnothing$ for $\delta> 1$, and $A_x^\delta= A_x^{-1}=\sphere^2$ for $\delta \le -1$. It is also clear that if $\delta_1\le \delta_2$, then  $A_x^{\delta_1}\subseteq A_x^{\delta_2}$ with a strict inclusion if $\delta_1< \delta_2$ and $\delta_1,\delta_2\in [-1,1]$. So while $\delta$ only has practical significance while taking values in $[-1,1]$, allowing it to take all values in $\R$ makes some of the upcoming proofs easier.

By Propositions \ref{eccdom} and \ref{eccplane}, we have the following statement.

\begin{prop}\label{geoprop}
Let $0<\delta<1$ and $B\subseteq A_x^\delta$. Then if $\epsilon>\frac1\delta$, $h_{x,\epsilon}[B] \subset H_\epsilon(x).$ In particular, $\frac1{\epsilon}<\delta$ implies that $A_x^\delta\subset D_\epsilon(x)$, $\frac1{\epsilon}>\delta$ implies that $D_\epsilon(x)\subset A_x^\delta$, and $\frac1{\epsilon}=\delta$ implies that $\textnormal{Int}(A_x^\delta)= D_\epsilon(x)$.
\end{prop}           

Also note the following definition. 

\begin{Def}\label{cone-def}
For an element $x\in\R^3$ and a set $A\subset\R^3$, let the set $C_{x,A}=\{at+x(1-t)|t\in[0,1],a\in A\}$ be the union of all line segments from $x$ to $A$ and $C_{x,A,\infty}=\{at+x(1-t)|t\in[0,\infty),a\in A\}$ be the union of all rays from $x$ that intersect $A$.
\end{Def}

\section{Convex Weak Solutions}

We now have the background to construct and proceed with our discussion of the weak solution. Keep in mind that this weak solution definition, apart from some minor differences in notation, is identical to the weak solution defined in \cite{KOT}. 

Let $c=\text{min}_{x,y\in T} \langle k_x,k_y\rangle$, $\ell=\text{min}_{x\in T} |x|$, and $L=\text{max}_{x\in T} |x|$. Assume we are given a set $T\subseteq \R^3$. We say that $T$ satisfies Hypothesis H1 if the following condition is met. 

\begin{h1*}
$T$ is a compact subset of $\R^3$ contained in a half space of $\R^3$, $\ell>0$, and $2\ell c>L.$
\end{h1*}

 Note that this is Hypothesis H1 from \cite{KOT}.
 
 We also define a constant,
\begin{equation}\label{e0}
    \epsilon_0 = \frac{\ell+\sqrt{\ell^2 -2L\ell c+L^2}}{2\ell c-L},
\end{equation}
that depends only on $T$.

We first assume we are given a target set $T$ that satisfies Hypothesis H1. Let $\tilde{H}_\epsilon(x)$ be the convex body bounded by ${H}_\epsilon(x).$ Consider the aperture $D_T^{\delta_\gamma}= \textnormal{Int}\left(\bigcap_{x\in T} A_x^{\delta_\gamma} \right)$ where $\delta_\gamma=\frac{1}{\epsilon_0+\gamma}$ for some $\gamma>0$. We then define a simply connected refractor over $D_T^{\delta_\gamma}$ as the boundary of the intersection of the convex bodies bounded by hyperboloids. Specifically,
\begin{equation}\label{refrdef3}
   R=\partial h \textnormal{ where }h=\bigcap_{x\in T}\tilde{H}_{\epsilon_x}(x)
\end{equation}
where each $\epsilon_x\ge\epsilon'\ge\epsilon_0+\gamma=\frac1{\delta_\gamma}$. Observe that 
\begin{equation}\label{equRconvex}
    R=\left\{m\sup_{x\in T}\h_{x,\epsilon_x}(m)\left|m\in \Int\left(\bigcap_{x\in T} A_x^{\frac1{\epsilon_x}}\right)\right.\right\}.
\end{equation}
Note that $D_T^{\delta_\gamma}\subseteq\Int\left(\bigcap_{x\in T} A_x^{\frac1{\epsilon_x}}\right)$, and $\sup_{x\in T}\h_{x,\epsilon_x}$ is twice differentiable almost everywhere by Alexandrov's theorem \cite{convex}; thus $R$ may be considered a refactor per Definition \ref{refractdef}. Let 
\begin{equation}\label{convexWeakSet}
    \mathcal{R}_{convex}^{\epsilon'}(T)
\end{equation}
be the set of all such refractors. Please note that by Lemma 1 in \cite{KOT}, the set $\mathcal{R}_{convex}^{\epsilon'}(T)$ is nonempty. 

Note the following definition.
\begin{Def}\label{supportdef}
A hyperboloid $H(x)$ is said to be \textbf{supporting} to a set $Q\subset \R^3$ at a point $z\in \partial Q$ if the convex body $\tilde{H}(x)$ bounded by $H(x)$ contains $Q$ and $z\in H(x)\cap \partial Q$.
\end{Def}  
For a subset $\omega \subseteq T$ and a refractor $R\in \mathcal{R}_{convex}^{\epsilon'}(T)$ put 
\begin{equation}
M(\omega)=\{z\in R| \textnormal{ there exists } x\in \omega \textnormal{ such that } H(x) \textnormal{ is supporting to } R \textnormal{ at }z\}.
\end{equation}
The intersection of $\overline{D_T^{\delta_\gamma}}$ with the image of the set $M(\omega)$ under radial projection on $\sphere^2$ we call the {\it visibility set of $\omega$} and denote it by $V_{convex} (\omega).$ By Lemma 4 of \cite{KOT}, this set $V_{convex} (\omega)$ is measurable for all Borel sets $\omega \subseteq T.$

For $m \in D_T^{\delta_\gamma}$ let $r(m)$ be the set of points of intersection between the refractor $R$ and the ray of direction $m$ originating at $\origin$. The possibly multivalued map $\alpha_{convex}:D_T^{\delta_\gamma} \to T,$
\begin{equation}
     \alpha_{convex}(m)=\{x\in T| \textnormal{ there exists } H(x) \textnormal{ supporting to } R\textnormal{ at } r(m)\}
\end{equation}
is called the {\it refractor map}.

Assume we are given a nonnegative $g\in L^1(\mathbb{S}^2)$. Let us define for measurable $X\subseteq \sphere^2$
\begin{equation}\label{s-meas-convex}
    \mu_{g}(X)=\int_X g(m) d\sigma(m)
\end{equation}
where $\sigma$ denotes the standard measure on $\mathbb{S}^2.$ Assume that $g\equiv 0$ outside of $D_T^{\delta_\gamma}$. 

In order to formulate and solve the refractor problem (in the framework of weak solutions to be defined below), we need to define a measure representing the energy generated by $g$ and redistributed by a refractor $R \in  \mathcal{R}_{convex}^{\epsilon'}(T ).$

Define for any refractor $R\in \mathcal{R}_{convex}^{\epsilon'}(T)$, 
\begin{equation}
    G_{convex}(\omega)=\mu_{g}(V_{convex}(\omega))
\end{equation}
which we will deem {\it the energy function}. It can be shown that $G$ is a finite measure on the Borel $\sigma$-algebra of $T$.

Let $F$ be a nonnegative, finite, Borel measure on Borel subsets of $T$. We say that a refractor $R\in \mathcal{R}_{convex}^{\epsilon'}(T)$ is a {\it convex weak solution to the refractor problem} if the refractor map $\alpha_{convex}$ determined by $R$ is such that $\alpha_{convex}(m)\subseteq T$ for all $m\in D_T^{\delta_\gamma}$, and 
\begin{equation}\label{weak3}
    F(\omega)=G_{convex}(\omega)\textnormal{ for any Borel set } \omega \subseteq T.
\end{equation}

\section{Uniqueness Theorems}

We start with some uniqueness results. Note that Theorem \ref{unique} can be considered as a direct corollary to Theorem \ref{unique-continuous}. We include both as separate statements and proofs; as discrete versions of the uniqueness theorems proved to be of special interest in related problems; see \cite{OK1} and \cite{OK2}. We proceed with the following lemma; which is shown in the proof of Lemma 2 in \cite{KOT}.

\begin{lemma}\label{closed-convex}
    Let $T$ be a target set that satisfies Hypothesis H1. Suppose we are given positive real numbers $\gamma$ and $\epsilon'$ such that $\epsilon'\ge \epsilon_0+\gamma$ where $\epsilon_0$ is defined by (\ref{e0}). Let $R\in \mathcal{R}_{convex}^{\epsilon'}(T)$. Then $V_{convex}(\omega)$ is closed for all closed $\omega\subseteq T$.
\end{lemma}

Before we proceed, note that if we write $V_{convex}(R;\omega)$ for some Borel set $\omega\subseteq T$ and some refractor $R\in \mathcal{R}_{convex}^{\epsilon'}(T)$, this is specifically the visibility set for the refractor $R$ evaluated on the set $\omega$. Similarly, if we write $G_{convex}(R;\omega)$ for some Borel set $\omega\subseteq T$ and some refractor $R\in \mathcal{R}_{convex}^{\epsilon'}(T)$, this is specifically the energy function for the refractor $R$ evaluated on the set $\omega$. We will be using this when we are talking about multiple refractors and we need to specify the energy function for each refractor.

Here we consider the case of the refractor problem (\ref{weak3}) where the set $T$ is finite. We prescribe the measure $F$ in (\ref{weak3}) as a Dirac measure concentrated at points in $T.$ We now introduce notation for refractors in the discrete case. Let $T=\{x_1,x_2,\dots,x_k\}$. We set $H_i = H(x_i)$ and the eccentricity of $H_i$ we denote by $\epsilon_i$. For the hyperboloids $H_1,\dots,H_k$ define the refractor
\begin{equation}
    R= \partial\left(\bigcap_{i=1}^k \tilde{H}_i\right)\in \mathcal{R}_{convex}^{\epsilon'}(T).
\end{equation}

Since each hyperboloid $H_i$ is uniquely defined by its eccentricity $\epsilon_i$, the refractor $R$ can be identified with the point with coordinates $(\epsilon_1,\epsilon_2,\dots,\epsilon_k)$ in the region
\begin{equation}
    \epsilon_1\ge\epsilon',\epsilon_2\ge\epsilon',\dots,\epsilon_k\ge\epsilon'
\end{equation}
in $k-$dimensional euclidean space. Thus we can write a refractor $R\in \mathcal{R}_{convex}^{\epsilon'}(T)$ as $(\epsilon_1, \epsilon_2,\dots,\epsilon_k).$ We start with a uniqueness theorem for the discrete case.

\begin{theorem}\label{unique}
    Let $T = \{x_1,\dots, x_k\}$ be a collection of $k$ distinct points that satisfy Hypothesis H1. Suppose we are given positive real numbers $\gamma$ and $\epsilon'$ such that $\epsilon'\ge \epsilon_0+\gamma$ where $\epsilon_0$ is defined by (\ref{e0}).  Assume we are given a nonnegative $g\in L^1(\mathbb{S}^2)$ such that $g>0$ inside $D_T^{\delta_\gamma}$ and $g\equiv 0$ outside $D_T^{\delta_\gamma}$ where $\delta_\gamma=\frac1{\epsilon_0+\gamma}$. Let $f_1,\dots,f_k$ be a collection of positive real numbers such that 
    \begin{equation}
        \sum_{i=1}^k f_i = \mu_{g}(D_T^{\delta_\gamma}).
    \end{equation}
    Let $\overline{R}=(\overline{\epsilon_1},\dots, \overline{\epsilon_k})$ and $\tilde{R}=(\tilde{\epsilon_1},\dots, \tilde{\epsilon_k})$ be refractors in $\mathcal{R}_{convex}^{\epsilon'}(T)$ such that $G_{convex}(\tilde{R};x_i)=G_{convex}(\overline{R};x_i)=f_i$ for all $i\in[k].$ 
    
    Then the inequality $\tilde{\epsilon_j}\ge \overline{\epsilon_j}$ for some $j$ implies that $\tilde{\epsilon_i}\ge \overline{\epsilon_i}$ for all $i\in[k].$ Furthermore, the equality $\tilde{\epsilon_j}= \overline{\epsilon_j}$ for some $j$ implies that $\tilde{\epsilon_i}= \overline{\epsilon_i}$ for all $i\in[k].$
\end{theorem}
\begin{proof}
Let $J$ be a nonempty subset of $[k]$ such that for any $i \in J,$ $\tilde{\epsilon_i} > \overline{\epsilon_i}$, and for any $i\in [k]\setminus J$, $\tilde{\epsilon_i} \le \overline{\epsilon_i}$. Note that $m\in V_{convex}(\tilde{R};\{x_i|i \in J \})$ if and only if there exists some $i\in J$ such that $\h_{x_i,\tilde{\epsilon_i}}(m)\ge \h_{x_\ell,\tilde{\epsilon_\ell}}(m)$ for all $\ell\in [k]\setminus J.$ For this $m$: since $\h_{x_i,\overline{\epsilon_i}}(m)>\h_{x_i,\tilde{\epsilon_i}}(m)$ for all $i\in J$ and $\h_{x_\ell,\overline{\epsilon_\ell}}(m)\le \h_{x_\ell,\tilde{\epsilon_\ell}}(m)$ for all $\ell\in [k]\setminus J$, then there exists some $i\in J$ such that $\h_{x_i,\overline{\epsilon_i}}(m)> \h_{x_\ell,\overline{\epsilon_\ell}}(m)$ for all $\ell\in [k]\setminus J$. Thus, any $m \in V_{convex}(\tilde{R};\{x_i|i \in J \})$ is an interior point of $V_{convex}(\overline{R};\{x_i|i \in J \})$; in other words, $V_{convex}(\tilde{R};\{x_i|i \in J \})\subseteq \Int(V_{convex}(\overline{R};\{x_i|i \in J \}))$. Recall that, by Lemma \ref{closed-convex}, $V_{convex}(\tilde{R};\{x_i|i \in J \})$ is closed and, since $f_1,\dots,f_k$ are positive, $V_{convex}(\overline{R};\{x_i|i \in J \})$ is nonempty. Then $\Int(V_{convex}(\overline{R};\{x_i|i \in J \}))\setminus V_{convex}(\tilde{R};\{x_i|i \in J \})$ is open and nonempty. So $\mu_g(V_{convex}(\overline{R};\{x_i|i \in J \})\setminus V_{convex}(\tilde{R};\{x_i|i \in J \}))>0$. Therefore we must have
    \begin{equation}
        \sum_{i\in J} f_i=G_{convex}(\tilde{R};\{x_i|i \in J \})<G_{convex}(\overline{R};\{x_i|i \in J \})=\sum_{i\in J} f_i
    \end{equation}
which is a contradiction because $G_{convex}(\tilde{R};x_i)=G_{convex}(\overline{R};x_i)=f_i$ for all $i\in[k].$ The theorem is proved.
\end{proof}

 Observe that for all refractors $R\in \mathcal{R}_{convex}^{\epsilon'}(T)$, there exists a function $K:T\to [\epsilon',\infty)$ such that $R=\partial(\bigcap_{x\in T} \tilde{H}_{K(x)}(x))$. Since each hyperboloid $\tilde{H}_{K(x)}(x)$ is uniquely determined by $K$, the refactor $R$ can be identified with the function $K:T\to [\epsilon',\infty)$. Thus we can write a refractor $R\in \mathcal{R}_{convex}^{\epsilon'}(T)$ as $[K]$ where $K:T\to[\epsilon', \infty)$; note that $[K]=\left\{m\sup_{x\in T}\h_{x,K(x)}(m)\left|m\in \Int\left(\bigcap_{x\in T} A_x^{\frac1{K(x)}}\right)\right.\right\}$. Given a refractor $R\in \mathcal{R}_{convex}^{\epsilon'}(T)$, we call $K:T\to[\epsilon', \infty)$ {\it the maximal function of $R$} if $R=\left\{m\max_{x\in T}\h_{x,K(x)}(m)\left|m\in \Int\left(\bigcap_{x\in T} A_x^{\frac1{K(x)}}\right)\right.\right\}$. We proceed with the following lemma. 

 \begin{lemma}\label{max-func-lemma}
     Let $R\in\mathcal{R}_{convex}^{\epsilon'}(T)$ be a refractor such that for all $x\in T$, $V_{convex}(\{x\})$ is nonempty. Then there exists a $K:T\to[\epsilon', \infty)$ that is a the maximal function of $R$.
 \end{lemma}
 \begin{proof}
 By Lemma 2 in \cite{KOT}, $T\subset \bigcap_{x\in T} \tilde{H}_{\epsilon_x}(x)$. Therefore, by Definition \ref{supportdef}, that $m\in V_{convex}(\{x\})$ if and only if there exists a corresponding $\epsilon_x'\ge\epsilon'$ such that $\h_{x,\epsilon_x'}(m) =\sup_{x\in T} \h_{x,\epsilon_x}(m).$ Define $K(x)=\epsilon_x'$ for all $x\in T.$ Then $R=\left\{m\max_{x\in T}\h_{x,K(x)}(m)\left|m\in \Int\left(\bigcap_{x\in T} A_x^{\frac1{K(x)}}\right)\right.\right\}$.
 \end{proof}

 We now conclude with a uniqueness theorem for more general measures and target sets.

\begin{theorem}\label{unique-continuous}
    Let $T $ be a target set that satisfies Hypothesis H1. Let $F$ be a nonnegative, finite, Borel measure on Borel subsets of $T.$ Suppose we are given positive real numbers $\gamma$ and $\epsilon'$ such that $\epsilon'\ge \epsilon_0+\gamma$ where $\epsilon_0$ is defined by (\ref{e0}).  Assume we are given a nonnegative $g\in L^1(\mathbb{S}^2)$ such that $g>0$ inside $D_T^{\delta_\gamma}$ and $g\equiv 0$ outside $D_T^{\delta_\gamma}$ where $\delta_\gamma=\frac1{\epsilon_0+\gamma}$ such that 
    \begin{equation}
        F(T) = \mu_{g}(D_T^{\delta_\gamma}).
    \end{equation}
    Let $\overline{R}$ and $\tilde{R}$ be refractors in $\mathcal{R}_{convex}^{\epsilon'}(T)$ such that for all nonempty Borel $\omega \subseteq T$: $F(\omega)=G_{convex}(\tilde{R};\omega)=G_{convex}(\overline{R};\omega)$, $V_{convex}(\tilde{R};\omega)\neq \varnothing$, and $V_{convex}(\overline{R};\omega)\neq \varnothing$. 
    
    Then there exists functions $\overline{K}:T\to[\epsilon', \infty)$ and $\tilde{K}:T\to[\epsilon', \infty)$ that are, respectively, maximal functions of $\overline{R}$ and $\tilde{R}$ such that the inequality $\tilde{K}(x)\ge \overline{K}(x)$ for some $x\in T$ implies that $\tilde{K}(y)\ge \overline{K}(y)$ for all $y\in T.$ Furthermore, the equality $\tilde{K}(x)= \overline{K}(x)$ for some $x\in T$ implies that $\tilde{K}(y)= \overline{K}(y)$ for all $y\in T.$
\end{theorem}

\begin{proof}
By Lemma \ref{max-func-lemma}, there exists functions $\overline{K}:T\to[\epsilon', \infty)$ and $\tilde{K}:T\to[\epsilon', \infty)$ that are maximal functions of $\overline{R}$ and $\tilde{R}$ respectively. Note that $\overline{R}=[\overline{K}]$ and $\tilde{R}=[\tilde{K}]$.

    Let $J$ be a nonempty closed subset of $T$ such that for any $x \in J,$ $\tilde{K}(x)> \overline{K}(x)$, and for any $x\in T\setminus J$, $\tilde{K}(x)\le \overline{K}(x)$. Note that $m\in V_{convex}(\tilde{R};J)$ if and only if there exists some $z\in J$ such that $\h_{z,\tilde{K}(z)}(m)\ge \h_{z',\tilde{K}(z')}(m)$ for all $z'\in T\setminus J.$ For this $m$: since $\h_{z,\overline{K}(z)}(m)>\h_{z,\tilde{K}(z)}(m)$ for all $z\in J$ and $\h_{z',\overline{K}(z')}(m)\le \h_{z',\tilde{K}(z')}(m)$ for all $z'\in T\setminus J$, then there exists some $z\in J$ such that $\h_{z,\overline{K}(z)}(m)> \h_{z',\overline{K}(z')}(m)$ for all $z'\in T\setminus J$. Thus, any $m \in V_{convex}(\tilde{R};J)$ is an interior point of $V_{convex}(\overline{R};J)$; in other words, $V_{convex}(\tilde{R};J)\subseteq \Int(V_{convex}(\overline{R};J))$. Recall that, by Lemma \ref{closed-convex}, $V_{convex}(\tilde{R};J)$ is closed. Then $\Int(V_{convex}(\overline{R};J))\setminus V_{convex}(\tilde{R};J)$ is open and nonempty. So $\mu_g(V_{convex}(\overline{R};J)\setminus V_{convex}(\tilde{R};J))>0$. Therefore we must have
    \begin{equation}
        F(J)=G_{convex}(\tilde{R};J)<G_{convex}(\overline{R};J)=F(J)
    \end{equation}
which is a contradiction because $F(\omega)=G_{convex}(\tilde{R};\omega)=G_{convex}(\overline{R};\omega)$ for all Borel $\omega\subseteq T.$ The theorem is proved.
\end{proof}

\section{Weak Solutions in the Discrete Case}

 Here we consider the case of the refractor problem (\ref{weak3}) where the set $T$ is finite. We prescribe the measure $F$ in (\ref{weak3}) as a Dirac measure concentrated at points in $T.$  We recall notation for refractors in the discrete case. Let $T=\{x_1,x_2,\dots,x_k\}$. We set $H_i = H(x_i)$ and the eccentricity of $H_i$ we denote by $\epsilon_i$. For the hyperboloids $H_1,\dots,H_k$ define the refractor
\begin{equation}
    R= \partial\left(\bigcap_{i=1}^k \tilde{H}_i\right)\in \mathcal{R}_{convex}^{\epsilon'}(T).
\end{equation}

Since each hyperboloid $H_i$ is uniquely defined by its eccentricity $\epsilon_i$, the refractor $R$ can be identified with the point with coordinates $(\epsilon_1,\epsilon_2,\dots,\epsilon_k)$ in the region
\begin{equation}
    \epsilon_1\ge\epsilon',\epsilon_2\ge\epsilon',\dots,\epsilon_k\ge\epsilon'
\end{equation}
in $k-$dimensional euclidean space. Thus we can write a refractor $R\in \mathcal{R}_{convex}^{\epsilon'}(T)$ as $(\epsilon_1, \epsilon_2,\dots,\epsilon_k).$ We now recall Theorem 9 from \cite{KOT}.

\begin{theorem}[Theorem 9 in \cite{KOT}]\label{KOTthrm}
Let $T = \{x_1,\dots, x_k\}$ be a collection of $k$ distinct points in $\R^3\setminus \{\origin\}$, $k > 2.$ Assume that $T$ satisfies Hypothesis H1. Let $\gamma$, $\epsilon_M$, $\epsilon_{min}$, and $\epsilon_{max}$ be positive real numbers such that $\epsilon_0 +\gamma <\epsilon_M \le \epsilon_{min} \le \epsilon_{max} < \infty$, where $\epsilon_0$ is defined by (\ref{e0}). Assume we are given a nonnegative $g\in L^1(\mathbb{S}^2)$ such that $g\equiv 0$ outside $D_T^{\delta_\gamma}$ where $\delta_\gamma=\frac1{\epsilon_0+\gamma}$. Let $f_1,\dots,f_k$ be nonnegative
real numbers such that

\begin{equation}
    \sum_{i=1}^k f_i = \mu_{g}(D_T^{\delta_\gamma}).
\end{equation}

Suppose that there also exists some $\ell \in [k]$ such that for all $i \in [k]$, $i \neq \ell$,
\begin{equation}
   G_{convex}(R_\ell; x_i) \le f_i
\end{equation}
where $R_\ell = (\epsilon_1 = \epsilon_{max},...,\epsilon_{\ell-1} = \epsilon_{max}, \epsilon_\ell = \epsilon_{min}, \epsilon_{\ell+1} = \epsilon_{max},...,\epsilon_k = \epsilon_{max}),$ and

\begin{equation}
   G_{convex}(R_{\ell i}; x_\ell) < f_\ell
\end{equation}
where $R_{\ell i} = (\epsilon_1 = \epsilon_{max},\dots,\epsilon_{i-1} = \epsilon_{max}, \epsilon_i = \epsilon_M , \epsilon_{i+1} = \epsilon_{max},\dots,\epsilon_{\ell-1} = \epsilon_{max}, \epsilon_\ell =
\epsilon_{min},\epsilon_{\ell+1} = \epsilon_{max},...,\epsilon_k = \epsilon_{max}).$ Then there exists a refractor $R = (\epsilon_1,\dots,\epsilon_k) \in \mathcal{R}_{convex}^{\epsilon_M} (T)$ such
that

\begin{equation}
G_{convex}(R; x_i) = f_i \textnormal{ for all } i \in[k].
\end{equation}
\end{theorem}

This theorem inspires an obvious corollary.

\begin{coro}\label{wsol3coro}
Let $T = \{x_1,\dots, x_k\}$ be a collection of $k$ distinct points in $\R^3\setminus \{\origin\}$, $k > 2.$ Assume that $T$ satisfies Hypothesis H1. Let $\gamma$, $\epsilon_M$, $\epsilon_{min}$, and $\epsilon_{max}$ be positive real numbers such that $\epsilon_0 +\gamma <\epsilon_M \le \epsilon_{min} \le \epsilon_{max} < \infty$, where $\epsilon_0$ is defined by (\ref{e0}). Assume we are given a nonnegative $g\in L^1(\mathbb{S}^2)$ such that $g\equiv 0$ outside $D_T^{\delta_\gamma}$ where $\delta_\gamma=\frac1{\epsilon_0+\gamma}$. Let $f_1,\dots,f_k$ be nonnegative
real numbers such that

\begin{equation}
    \sum_{i=1}^k f_i = \mu_{g}(D_T^{\delta_\gamma}).
\end{equation}

Suppose that there also exists some $\ell \in [k]$ where $f_\ell>0$ such that for all $i \in [k]$, $i \neq \ell$,

\begin{equation}
   G_{convex}(R_\ell; x_i)=0
\end{equation}
where $R_\ell = (\epsilon_1 = \epsilon_{max},...,\epsilon_{\ell-1} = \epsilon_{max}, \epsilon_\ell = \epsilon_{min}, \epsilon_{\ell+1} = \epsilon_{max},...,\epsilon_k = \epsilon_{max}),$ and

\begin{equation}
   G_{convex}(R_{\ell i}; x_\ell) =0
\end{equation}
where $R_{\ell i} = (\epsilon_1 = \epsilon_{max},\dots,\epsilon_{i-1} = \epsilon_{max}, \epsilon_i = \epsilon_M , \epsilon_{i+1} = \epsilon_{max},\dots,\epsilon_{\ell-1} = \epsilon_{max}, \epsilon_\ell =
\epsilon_{min},\epsilon_{\ell+1} = \epsilon_{max},...,\epsilon_k = \epsilon_{max}).$ Then there exists a refractor $R = (\epsilon_1,\dots,\epsilon_k) \in \mathcal{R}_{convex}^{\epsilon_M} (T)$ such
that

\begin{equation}
G_{convex}(R; x_i) = f_i \textnormal{ for all } i \in[k].
\end{equation}
\end{coro}

We will now use the above corollary to prove the following proposition. 

\begin{prop}\label{line-conv}
Let $T = \{x_1,\dots, x_k\}$ be a collection of $k$ distinct points in $\R^3\setminus \{\origin\}$ such that $T$ satisfies Hypothesis H1 and $\min_{x,y\in T}\langle k_x,k_y\rangle=1.$ Suppose that we are given $\gamma>0$ such that $\epsilon_0+\gamma<\lim_{t\to K^+}\frac{1}{t-1}$ for $K=\frac{\max_{x\in T} |x|}{\min_{x\in T} |x|}$ and $\epsilon_0$ is defined by (\ref{e0}). Assume we are given a nonnegative $g\in L^1(\mathbb{S}^2)$ such that $g\equiv 0$ outside $D_T^{\delta_\gamma}$ where $\delta_\gamma=\frac1{\epsilon_0+\gamma}$. Let $f_1,\dots,f_k$ be nonnegative
real numbers such that
\begin{equation}
    \sum_{i=1}^k f_i = \mu_{g}(D_T^{\delta_\gamma})
\end{equation}
and for the $\ell\in[k]$ where $|x_\ell|=\max_{y\in T}|y|$, $f_\ell>0.$ 

Then there exists an $\epsilon_M\in(\epsilon_0+\gamma,\lim_{t\to K^+}\frac{1}{t-1})$ such that we can construct a convex, rotationally symmetric refractor $R = (\epsilon_1,\dots,\epsilon_k) \in \mathcal{R}_{convex}^{\epsilon_M} (T)$ where
\begin{equation}
G_{convex}( R;x_i) = f_i \textnormal{ for all } i \in[k].
\end{equation}
\end{prop}

\begin{proof}

 Note that $\max_{x,y}\langle k_x,k_y\rangle=1$ implies that $k_x=k_y$ for all $x,y\in T,$ and $\epsilon_0=1$ as defined by (\ref{e0}). The case where $k=1$ is trivial; let $k\ge 2$. Assume that $|x_i|\ge|x_{i+1}|$ for all $i\in[k-1]$ and thus $f_1>0.$ Recall that for $x\in T$ by Proposition \ref{eccplane}, $\h_{\epsilon,x}(m)\to \frac{|x|}{2\langle m,k_x\rangle}$ as $\epsilon\to \infty$ for $m\in D_T^{\delta_\gamma}$.

Observe that
\begin{equation}
    \frac{|x_1|}{2\langle m,k_x\rangle}<\frac{|x_k|(1-\epsilon_M^2)}{2\epsilon_M(1-\epsilon_M\langle m , k_{x}\rangle)}\textnormal{ for }m\in  D_T^{\delta_\gamma},
\end{equation}
if and only if
\begin{equation}
    \frac{|x_1|}{2}<\frac{|x_k|(1-\epsilon_M^2)}{2\epsilon_M(1-\epsilon_M)}.
\end{equation}
Thus we have that $\epsilon_M<\frac1{K-1}$ where $K=\frac{|x_1|}{|x_k|}$. Note that by Hypothesis H1 and the fact that $k\ge2$, we have $1<K<2$ and $1<\frac1{K-1}<\infty$. Thus we can have that $1=\epsilon_0<\epsilon_0+\gamma<\epsilon_M<\frac1{K-1}.$ If $k=2$, by the continuity implied by Lemma 8 of \cite{KOT}, there exists a refractor $R = (\epsilon_1,\epsilon_2) \in \mathcal{R}_{convex}^{\epsilon_M} (T)$ such
that $G_{convex}(R; x_i) = f_i \textnormal{ for all } i \in[k].$

If $k>2,$ we borrow language and notation from Corollary \ref{wsol3coro}. By continuity, if $\epsilon_{min}=\epsilon_{max}$ be sufficiently large such that $\frac1{K-1}<\epsilon_{min}=\epsilon_{max}$, then, assuming that $\epsilon_M<\frac1{K-1}$, $G_{convex}(R_1;x_i)=0$ and $G_{convex}(R_{1i};x_1)=0$ for all $i\in[k]$ such that $i\neq 1$. Therefore by Corollary \ref{wsol3coro}, there exists a refractor $R = (\epsilon_1,\dots,\epsilon_k) \in \mathcal{R}_{convex}^{\epsilon_M} (T)$ such
that $G_{convex}(R; x_i) = f_i \textnormal{ for all } i \in[k].$

\end{proof}

The above proposition motivates our main result.

\begin{theorem}\label{convRefr}
Assume that we are given some $w,W\in(0,\infty)$ where $w>\frac{W}2$. Given some $m_*\in \sphere^2$, let 
\begin{equation}
    S(m_*,\xi)=\{x\in\R^3|w\le|x|\le W, \langle k_x,m_*\rangle\ge1-\xi \}
\end{equation}
where $1-\cos\left(\frac12\arccos\left(\frac{W}{2w}\right)\right)>\xi>0$; note that $S(m_*,\xi)$ satisfies Hypothesis H1. Let $T = \{x_1,\dots, x_k\}\subset S(m_*,\xi)$ be a collection of $k$ distinct points. Recall that $\delta_\gamma=\frac1{\epsilon_0+\gamma}$ where $\gamma>0$ and $\epsilon_0$ is defined by (\ref{e0}).

Then there exists positive $\xi$ and $\gamma$ such that, for any nonnegative $g\in L^1(\mathbb{S}^2)$ such that $g\equiv 0$ outside $D_T^{\delta_\gamma}$ and any collection $f_1,\dots,f_k$ of nonnegative real numbers where
\begin{equation}
    \sum_{i=1}^k f_i = \mu_{g}(D_T^{\delta_\gamma})
\end{equation}
and $f_\ell>0$ for the $\ell\in[k]$ where $|x_\ell|=\max_{y\in T}|y|$, there exists an $\epsilon_M>\epsilon_0+\gamma$ such that we can construct a refractor $R = (\epsilon_1,\dots,\epsilon_k) \in \mathcal{R}_{convex}^{\epsilon_M} (T)$ where
\begin{equation}
G_{convex}(R; x_i) = f_i \textnormal{ for all } i \in[k].
\end{equation}
\end{theorem}

\begin{proof}
Observe that $\xi\to 0$ implies that $\min_{x,y\in T}\langle k_x,k_y\rangle\to 1$. Assume that $|x_\ell|=\max_{y\in T}|y|$. Let $\h_{ Tmax,\epsilon}(m)=\max_{x\in T}\h_{x,\epsilon}(m)$ and $P_{Tmax}(m)=\max_{x\in T}\frac{|x|}{2\langle m,k_{x}\rangle}$ where $m\in D_T^{\delta_\gamma}.$ Note that $\h_{ Tmax,\epsilon}(m)\to \h_{x_\ell,\epsilon}(m)$ and $P_{ Tmax}(m)\to \frac{|x_\ell|}{2\langle m,k_{x_\ell}\rangle}$ as $\min_{y\in T}\langle k_{x_\ell},k_y\rangle\to 1$.

We borrow language and notation from Corollary \ref{wsol3coro}. Choose an $\epsilon_{min}$; then, by continuity, there exists a $\xi>0$ such that if $\min_{x,y\in T}\langle k_x,k_y\rangle\ge1-\xi$, then we have $P_{Tmax}(m)<\h_{x_\ell,\epsilon_{min}}(m)$ for all $m\in D_T^{\delta_\gamma}.$ Therefore by continuity there exists an $\epsilon_{max}$ such that $P_{T max}(m) < \h_{T max,\epsilon_{max}}(m) < \h_{x_\ell,\epsilon_{min} }(m)$ for all $m \in D_T^{\delta_\gamma}.$

Observe that  $\epsilon_0\to 1$ as $\min_{x,y\in T}\langle k_x,k_y\rangle\to 1$ and recall that a degenerate hyperboloid has eccentricity $1.$ Then, for sufficiently small $\gamma>0,$ by continuity we can choose $\epsilon_0+\gamma<\epsilon_M<\epsilon_{min}<\epsilon_{max}$, such that $P_{Tmax}(m)<\h_{x_\ell,\epsilon_{max}}(m)<\h_{Tmax,\epsilon_{min}}(m)<\h_{x_\ell,\epsilon_{M}}(m)$ for all $m\in D_T^{\delta_\gamma}$. 

Then $G_{convex}(R_\ell;x_i)=0$ and $G_{convex}(R_{\ell i};x_\ell)=0$ for all $i\in[k]$ such that $i\neq \ell$. Then by Corollary \ref{wsol3coro}, there exists a refractor $R = (\epsilon_1,\dots,\epsilon_k) \in \mathcal{R}_{convex}^{\epsilon_M} (T)$ such that $G_{convex}(R; x_i) = f_i \textnormal{ for all } i \in[k].$

\end{proof}

\section{Weak Solutions in the General Case}
In this section, we extend the results of Theorems \ref{convRefr} and \ref{unique} to the case of more general sets $T$ and
energy distributions $F$. We consider the case where prescribe the measure $F$ as a Lebesgue measure over $T$, specifically
\begin{equation}\label{lebesgue}
    F(\omega)=\int_\omega f(x)d\lambda(x)\textnormal{ for any Borel set }\omega\subseteq T
\end{equation}
for some given nonnegative function $f \in L^1(T )$; here $\lambda$ is the Lebesgue measure on $T$. 

\begin{theorem}\label{continuous}
Assume that we are given some $w,W\in(0,\infty)$ where $w>\frac{W}2$. Given some $m_*\in \sphere^2$, let 
\begin{equation}
    S(m_*,\xi)=\{x\in\R^3|w\le|x|\le W, \langle k_x,m_*\rangle\ge1-\xi \}
\end{equation}
where $1-\cos\left(\frac12\arccos\left(\frac{W}{2w}\right)\right)>\xi>0$; note that $S(m_*,\xi)$ satisfies Hypothesis H1. Let $T \subseteq S(m_*,\xi)$ be a closed set. Recall that $\delta_\gamma=\frac1{\epsilon_0+\gamma}$ where $\gamma>0$ and $\epsilon_0$ is defined by (\ref{e0}).  

Then there exists positive $\xi$ and $\gamma$ such that, for any nonnegative $f\in L^1(T)$ with a measure $F$ defined by (\ref{lebesgue}), and any nonnegative $g\in L^1(\mathbb{S}^2)$ where $g\equiv 0$ outside $D_T^{\delta_\gamma}$ where
\begin{equation}
    F(T) = \mu_{g}(D_T^{\delta_\gamma}),
\end{equation}
 there exists an $\epsilon_M>\epsilon_0+\gamma$ such that we can construct a convex refractor $R \in \mathcal{R}_{convex}^{\epsilon_M} (T)$ where $R$ that is a convex weak solution to the refractor problem (\ref{weak3}).
\end{theorem}

The following argument is based on similar arguments made in \cite{KOT}, \cite{para}, and \cite{OK1}; specifically, the argument made for Theorem 13 in \cite{KOT}. Even though Theorem 13 in \cite{KOT} is incorrect\footnote{It should be noted that in \cite{KOT}, due to erroneous assumptions, there were issues with attempts to construct refractors in the special case explored by Theorems 12 and 13; see Section \ref{app:wrongKOT}.}, the type of argument presented in its proof is broadly applicable; thus it would be good to put the argument in a correct context. 

\begin{proof}
The following argument follows the proof of Theorem 13 in \cite{KOT} very closely with some adjustments to fit into this new context. For the sake of the following proof, recall that our definition of the energy function can also be considered as a measure of $T$.

If $\mu_g(D_T^{\delta_\gamma})=0$, then any refractor $R \in \mathcal{R}_{convex}^{\epsilon_M} (T)$ will do. Assume that $\mu_g(D_T^{\delta_\gamma})>0$.

Since $T$ is bounded, for any $\delta > 0$ there exists an $N\in \mathbb{N}$ such that for each $k\ge N$ there exists a partition of $T$ into $k$ Borel sets $\omega_1^k,\dots,\omega_k^k$ such that
\begin{equation}
    \textnormal{diam}(\omega_i^k)\le \delta \textnormal{ for any } k \ge N, i\in[k].
\end{equation}
For each $k \in \mathbb{N}$, we choose an $x_i^k\in \omega_i^k$ for $i \in [k]$, and put
\begin{equation}
    F_i^k=F(\omega_i^k).
\end{equation}
Define a measure $F^k$ on $T$ by
\begin{equation}
    F^k(\omega)=\sum_{x_i^k\in\omega}F_i^k \textnormal{ for any Borel set } \omega \subseteq T.
\end{equation}
Note that $F^k$ converges weakly to $F$ as $k \to \infty.$ For each $k,$ there exists a nonempty $S\subseteq[k]$ such that $F_i^k>0$ for all $i\in S$. Since $\{x_i^k\}_{i\in S}$ and $\{F_i^k\}_{i\in S}$ satisfies the assumptions of Theorem \ref{convRefr}, there exists a convex refractor $R^k\in \mathcal{R}_{convex}^{\epsilon_M}(\{x_i^k|i\in S\})\subseteq\mathcal{R}_{convex}^{\epsilon_M}(\{x_i^k|i\in [k]\})\subseteq \mathcal{R}_{convex}^{\epsilon_M}(T)$ defined by hyperboloids with an eccentricity greater than or equal to some $\epsilon_M>\epsilon_0+\gamma$ such that
\begin{equation}
    G_{convex}(R^k;x_i^k ) = F_i^k \textnormal{ for } i \in[k].
\end{equation}

Let $G^k$ be the measure on $T$ defined by
\begin{equation}
    G^k(\omega) =\sum_{x_i^k\in \omega}G_{convex}(R^k;x_i^k ) 
\end{equation}
then obviously $F^k \equiv G^k$ for all $k\in \mathbb{N}$ and consequently, $G^k \to F$. To finish the
proof we need to construct a refractor $R$ whose energy function, $G$, would be the limit of measures $G^k$. This refractor is constructed in the following manner as a limit of refractors $R^k$.

First, we note that since $g \equiv 0$ outside $D_T^{\delta_\gamma}$, we only need to consider the part of the refractor $R^k\cap C_{\origin,D_T^{\delta_\gamma},\infty}$. See Definition \ref{cone-def} as a refresher on the meaning of $C_{\origin,D_T^{\delta_\gamma},\infty}$. Also for some $\epsilon_M > \epsilon_0+\gamma$ one can show that for any $R \in \mathcal{R}^{\epsilon_M} (T )$
\begin{equation}\label{incl}
    \left(R\cap C_{\origin,D_T^{\delta_\gamma},\infty}\right) \subseteq \textnormal{B}(\origin, b) \textnormal{ for some } b >0
\end{equation}
where $\textnormal{B}(\origin, b)$ is the open ball centered at the origin $\origin$ of radius $b$. Let us prove this statement. We define
\begin{equation}
    b = \max_{x\in T,m\in \overline{D_T}^{\delta_\gamma}}\h_{x,\epsilon_M}(m).
\end{equation}

Since $\h_{x,\epsilon_M}$ is a continuous function and $\overline{D_T}^{\delta_\gamma}$ is compact, this definition is correct and $b < \infty.$ Thus for any $\epsilon > \epsilon_M$, $\h_{x,\epsilon}(m)\le b$ and (\ref{incl}) is proved.
 
For each of the refractors $R^k$ we consider a bounded convex body
\begin{equation}
    h_b^k =h^k\cap C_{\origin,D_T^{\delta_\gamma},\infty}\cap\textnormal{B}(\origin, b)
\end{equation}
where for each $k\in \mathbb{N}$ the set $h^k$ is defined by (\ref{refrdef3}). By Blaschke’s selection theorem \cite{schneider_2013}, there exists a subsequence of $\{h_b^k\}$ which we again denote by $\{h_b^k\}$, which converges to some convex body $h_b.$

We show now that for each point $r \in [\partial( h_b) \cap C_{\origin,D_T^{\delta_\gamma},\infty}]\setminus \partial(\textnormal{B}(\origin, b))$ there exists a
hyperboloid $H^r(x)$ which is supporting to $h^b$ at point $r.$

Let $r \in [\partial( h_b) \cap C_{\origin,D_T^{\delta_\gamma},\infty}]\setminus \partial(\textnormal{B}(\origin, b)).$ Then there exists a sequence $\{r_k\}$ that converges to $r$ where each $r_k \in R^k$. Let $H(x_k)$ be a supporting hyperboloid to $R^k$ at $r_k$. Since $T$ is compact, $\{x_k\}$ contains a subsequence, which we will denote by $\{x_k^*\}$, converging to some $x \in T$. The convex body $\tilde{H} ( x_k^*)$ bounded by $H (x_k^*)$ contains the body $h_b^k$. The corresponding sequence $\{\tilde{H} (x_k^*)\}$ converges to the body $\tilde{H}^r(x)$ containing $h_b$ and $h_b^k$ converges to $h_b$. Therefore $\tilde{H}^r(x)$ contains $\partial (h_b)$. It follows that $H^r(x)$ is supporting to $h_b$ at $r$.

We now define the refractor $R = \partial(\bigcap_{x\in T}\tilde{H}^r(x))$ and show that the sequence of measures $G^k$, that are equivalent to the energy functions corresponding to the refractors $R^k$, converges weakly to the measure $G$, which is the energy function of the refractor $R$.

Let $\alpha_{convex}^k$ and $\alpha_{convex}$ be the refractor maps corresponding to $R^k$ and $R$ respectively. By a theorem of Reidemeister on singularities on convex surfaces (see \cite{schneider_2013}) the refractor maps $\alpha_{convex}^k$ for $k \in \mathbb{N}$ and $\alpha_{convex}$ are single-valued functions almost everywhere. Furthermore, for almost all $m\in D_T^{\delta_\gamma}$ the hyperboloids $H_k$ supporting to $R^k$ at points $r_k(m)$ converge to the hyperboloid $H$ supporting to $R$ at the point $r(m)$. Thus, $\alpha_{convex}^k(m)$ converges to $\alpha_{convex}(m)$ almost everywhere. 

If given a set of cardinality one, $\{z\}$, let $\textnormal{Ele}(\{z\})=z$. Let $Y^k(m)=\{x\in \alpha_{convex}^k(m)|x\in \{x_i^k\}_{i\in [k]}\}$ and let $J^k(m)\subseteq [k]$ be the set of indices such that $\{x_i^k|i\in J^k(m)\}=Y^k(m)$.
Let $z\in T$,
\begin{equation}
K^k(m)= x_{\min J^k(m)}^k,
\end{equation}
and 
\begin{equation}
K(m)= \begin{cases}
        \textnormal{Ele}(\alpha_{convex}(m)) & \text{if } |\alpha_{convex}(m)|=1\\
        z & \text{if } |\alpha_{convex}(m)|>1
    \end{cases}.
\end{equation}

Then for any continuous function $u$ on $T$ we have
\begin{equation}
    \int_T u dG^k=\int_{D_T^{\delta_\gamma}}u(K^k(m))d\mu_{g}(m)\longrightarrow \int_{D_T^{\delta_\gamma}}u(K(m))d\mu_{g}(m)=\int_T u dG
\end{equation}
as $k \to \infty$, that is, the measures $\{G^k\}$ converge weakly to $G$.
\end{proof}

\section{Rotationally Symmetric Convex Refractors on the Surface of a Right Circular Cone}

Both Theorems \ref{continuous} and \ref{line-conv} inspire this next result. 

\begin{coro}\label{line-conv-continuous}
Let $T$ be a closed set that satisfies Hypothesis H1 and $\min_{x,y\in T}\langle k_x,k_y\rangle=1$. Assume that we are given $\gamma>0$ such that $\epsilon_0+\gamma<\lim_{t\to K^+}\frac{1}{t-1}$ for $K=\frac{\max_{x\in T} |x|}{\min_{x\in T} |x|}$ and $\epsilon_0$ is defined by (\ref{e0}). Also, assume we are given a nonnegative $g\in L^1(\mathbb{S}^2)$ such that $g\equiv 0$ outside $D_T^{\delta_\gamma}$ where $\delta_\gamma=\frac1{\epsilon_0+\gamma}$. Suppose we are given a nonnegative $f\in L^1(T)$ and a measure $F$ defined by (\ref{lebesgue}) such that
\begin{equation}
    F(T) = \mu_{g}(D_T^{\delta_\gamma}).
\end{equation}

Then there exists an $\epsilon_M\in(\epsilon_0+\gamma,\lim_{t\to K^+}\frac{1}{t-1})$ such that we can construct a convex rotationally symmetric refractor $R  \in \mathcal{R}_{convex}^{\epsilon_M} (T)$ where $R$ that is a convex weak solution to the refractor problem (\ref{weak3}).    
\end{coro}

We will use the above corollary to create rotationally symmetric refractors with the target set on a right circular cone.

\begin{Def}\label{circlinedef}
Let $k\ge 2$, $d>0$, $1>\xi>0$, and $m_*,m'\in \sphere^2$ such that $\langle m_*,m'\rangle=0$. We create a Cartesian coordinate system centered at $\origin$ where $m_*$ is the direction of our $z$-axis, $m'$ is the direction of our $x$-axis, and $m_*\times m'$ is the direction of our $y$-axis. Let $(x,y,z)'$ represent a point in this system. 

Recall that, given a point $(x,y,z)'\in \R^3$, there exists $r \in [0, \infty)$, $\phi \in [0, \pi]$, $\theta \in [0, 2\pi)$, such that
\begin{align}
    x&=r\cos\theta \sin\phi\\
    y&=r\sin\theta\sin\phi\\
    z&=r\cos\phi.
\end{align}

Let $Q$ be a closed subset of the interval $(0,\infty)$. Define the set of points $T_{k,Q}^\xi(m_*,m')$ as
\begin{equation}
    \left\{\left(d\cos\left(\frac{2\pi j}k\right)\sin\left(\arccos(\xi)\right),d\sin\left(\frac{2\pi j}k\right)\sin\left(\arccos(\xi)\right),d\xi\right)'|j\in I, d\in Q\right\}
\end{equation}
where $I=\{0,1,\dots,k-1\}.$

Define the set $T_{\infty,Q}^\xi(m_*,m')$ as
\begin{equation}
    \left\{\left(d\cos\left(\theta\right)\sin\left(\arccos(\xi)\right),d\sin\left(\theta\right)\sin\left(\arccos(\xi)\right),d\xi\right)'|\theta\in [0,2\pi), d\in Q\right\}.
\end{equation}
\end{Def}

\begin{prop}\label{cov-poly}
    Let $m_*,m'\in \sphere^2$ such that $\langle m_*,m'\rangle=0$. Let $Q$ be a closed subset of the interval $(0,\infty)$, $k\ge 2$, and $1>\xi>0$ such that $T=T_{k,Q}^\xi(m_*,m')$ satisfies Hypothesis H1 such that $\epsilon_0<\lim_{t\to K^+}\frac{1}{t-1}$ for $K=\frac{\max_{x\in T} |x|}{\min_{x\in T} |x|}$ where $\epsilon_0$ is defined by (\ref{e0}).  Assume that we are given $\gamma>0$ such that $\epsilon_0+\gamma<\lim_{t\to K^+}\frac{1}{t-1}$. Also, assume we are given a nonnegative $g_*\in L^1(\mathbb{S}^2)$ that is rotationally symmetric about the axis defined by the ray of direction $m_*$ originating at $\origin.$ Let the function $g\in L^1(\mathbb{S}^2)$ be defined as $g\equiv g_*$ inside $D_T^{\delta_\gamma}$ and $g\equiv 0$ outside $D_T^{\delta_\gamma}$ where $\delta_\gamma=\frac1{\epsilon_0+\gamma}$. 

    Suppose we are given a nonnegative $f\in L^1(T)$ such that for every $d\in Q$:
    \begin{equation}
        f\left(\left(d\cos\left(\frac{2\pi j}k\right)\sin\left(\arccos(\xi)\right),d\sin\left(\frac{2\pi j}k\right)\sin\left(\arccos(\xi)\right),d\xi\right)'\right)
    \end{equation}
    is constant for all $j\in \{0,1,\dots,k-1\}$. Let $F$ be the measure defined by (\ref{lebesgue}) and
    \begin{equation}
    F(T) = \mu_{g}(D_T^{\delta_\gamma}).
\end{equation}

Then there exists an $\epsilon_M\in(\epsilon_0+\gamma,\lim_{t\to K^+}\frac{1}{t-1})$ such that we can construct a convex refractor $R  \in \mathcal{R}_{convex}^{\epsilon_M} (T)$ where $R$ that is a convex weak solution to the refractor problem (\ref{weak3}).    
\end{prop}

\begin{proof}
 We create a Cartesian coordinate system centered at $\origin$ where $m_*$ is the direction of our $z$-axis, $m'$ is the direction of our $x$-axis, and $m_*\times m'$ is the direction of our $y$-axis. Let $(x,y,z)'$ represent a point in this system. 

Recall that, given a point $(x,y,z)'\in \R^3$, there exists $r \in [0, \infty)$, $\phi \in [0, \pi]$, $\theta \in [0, 2\pi)$, such that
\begin{align}
    x&=r\cos\theta \sin\phi\\
    y&=r\sin\theta\sin\phi\\
    z&=r\cos\phi.
\end{align}

Let $T_j= \left\{\left(d\cos\left(\frac{2\pi j}k\right)\sin\left(\arccos(\xi)\right),d\sin\left(\frac{2\pi j}k\right)\sin\left(\arccos(\xi)\right),d\xi\right)'|d\in Q\right\}$. Note that $T=\bigcup_{i\in \{0,1,\dots,k-1\}} T_i.$

    Let $m_1,m_2\in\sphere^2$ such that $\langle m_1,m_2\rangle >-1$ and $\mathscr{L}(m_1,m_2)$ be the shortest arc on $\sphere^2$ between the points $m_1$ and $m_2$. For all $j\in \{0,1,\dots,k-1\}$, let $B_j\subset \overline{D_T^{\delta_\gamma}}$ be the set $\{t\in\mathscr{L}(m_*,y)|y\in \partial(D_T^{\delta_\gamma})\cap\partial(A_x^{\delta_\gamma})\}$ where $x\in T_j$. For $d\in Q$, since $T_{k,\{d\}}^\xi(m_*,m')$ defines the points of a regular $k$-gon centered at the axis defined by the ray of direction $m_*$ originating at $\origin,$ then $\mu_g(B_j)=\frac{\mu_g(D_T^{\delta_\gamma})}k$ for all $j$. For all $j\in \{0,1,\dots,k-1\}$, let $g_j$ be a function over $\sphere^2$ such that $g_j\equiv g$ inside $\Int(B_j)$ and $g_j\equiv 0$ outside $\Int(B_j)$. Note that $\mu_g(\Int(B_j))=\mu_{g_j}(D_{T_j}^{\delta_\gamma})$ for all $j$.
    
    Let $m^j\in \sphere^2$ be the unit vector such that a ray originating from $\origin$ of direction $m^j$ intersects every point in $T_j.$ Also, let $f_j$ be the restriction of the function $f$ to the set $T_j$ and $F_j$ be the corresponding measure as defined by (\ref{lebesgue}). By Proposition \ref{line-conv-continuous}, for all $j\in \{0,1,\dots,k-1\}$, there exists a refractor $R_j=\partial\left(\bigcap_{d\in Q} \tilde{H}_{\epsilon_d}(dm^j)\right) \in\mathcal{R}_{convex}^{\epsilon_M} (T_j)$, that is rotationally symmetric about the axis defined by the ray starting at $\origin$ with direction $m^j$, such that $R_j$ that is a convex weak solution to the refractor problem (\ref{weak3}); i.e. where $F_j(\omega) = \mu_{g_j}(V(R_j; \omega))$ for all Borel $\omega\subseteq T$. 
    
    Since for $x\in R_j$, we have that $|x|$ increases as $\langle k_x,m^j\rangle$ decreases for all $j\in \{0,1,\dots,k-1\}$, then for $j\in \{0,1,\dots,k-1\}$, defining $r_j(m)$ as the point of intersection between $R_j$ and the ray originating from $\origin$ in direction $m$, we have $|r_j(m)|=\max_{i\in \{0,1,\dots,k-1\}}|r_i(m)|$ for all $m\in B_j$.

Thus
\begin{equation}
    \partial\left(\bigcap_{j\in\{0,\dots,k-1\}}\left(\bigcap_{d\in Q} \tilde{H}_{\epsilon_d}(dm^j)\right)\right) \in\mathcal{R}_{convex}^{\epsilon_M} (T)
\end{equation}
is our refractor. 
\end{proof}

With an argument similar to that we use in the proof of Theorem \ref{continuous}, we obtain the following result from Proposition \ref{cov-poly}. 

\begin{theorem}\label{cov-rot}
    Let $m_*,m'\in \sphere^2$ such that $\langle m_*,m'\rangle=0$. Let $1>\xi>0$ and $Q$ be a closed subset of the interval $(0,\infty)$ such that $T=T_{\infty,Q}^\xi(m_*,m')$ satisfies Hypothesis H1 where $\epsilon_0<\lim_{t\to K^+}\frac{1}{t-1}$ for $K=\frac{\max_{x\in T} |x|}{\min_{x\in T} |x|}$ and $\epsilon_0$ is defined by (\ref{e0}). Assume that we are given $\gamma>0$ such that $\epsilon_0+\gamma<\lim_{t\to K^+}\frac{1}{t-1}$. Also, assume we are given a nonnegative $g\in L^1(\mathbb{S}^2)$ that is rotationally symmetric about the axis defined by the ray of direction $m_*$ originating at $\origin$ such that $g\equiv 0$ outside $D_T^{\delta_\gamma}$ where $\delta_\gamma=\frac1{\epsilon_0+\gamma}$. 
    
   Assume we have a nonnegative $f\in L^1(T)$ such that for every $d\in Q$:
    \begin{equation}
f\left(\left(d\cos\left(\theta\right)\sin\left(\arccos(\xi)\right),d\sin\left(\theta\right)\sin\left(\arccos(\xi)\right),d\xi\right)'\right)
    \end{equation}
    is constant for all $\theta\in [0,2\pi)$. Let $F$ be the measure defined by (\ref{lebesgue}) and
    \begin{equation}
    F(T) = \mu_{g}(D_T^{\delta_\gamma}).
\end{equation}

Then there exists an $\epsilon_M\in(\epsilon_0+\gamma,\lim_{t\to K^+}\frac{1}{t-1})$ such that we can construct a convex, rotationally symmetric refractor $R  \in \mathcal{R}_{convex}^{\epsilon_M} (T)$ where $R$ that is a convex weak solution to the refractor problem (\ref{weak3}).    
\end{theorem}

\section{No Set of Points Satisfies Both Hypotheses H1 and H2}\label{app:wrongKOT}

In the paper of Kochengin et al. \cite{KOT}, they specify two assumptions with regards to the target set $T$ that they call Hypothesis H1 and Hypothesis H2. In Lemma 11, Theorem 12, and Theorem 13, they prove key results with the assumption that both hypothesis hold for $T$. In this section we will prove that Hypotheses H1 and H2 are inherently contradictory. We start by restating some key definitions from \cite{KOT} and proceed with the proof.  

Let $k_x=\frac{x}{|x|}$. We are given a set of points $T$ in $\R^3\setminus\{\origin\}$, where $c=\text{min}_{x,y\in T} \langle k_x,k_y\rangle$, $\ell=\text{min}_{x\in T} |x|$, and $L=\text{max}_{x\in T} |x|$.

We also use the definition of $\epsilon_0$ provided in Lemma 2 of \cite{KOT}:

\begin{equation}\label{KOT:e0}
    \epsilon_0 = \frac{\ell+\sqrt{\ell^2 -2L\ell c+L^2}}{2\ell c-L}.
\end{equation}

In \cite{KOT} we have Hypothesis H1 which is as follows.

\begin{h1*}
$T$ is a compact subset of $\R^3$ contained in a half space of $\R^3$, $\ell>0$, and $2\ell c>L.$
\end{h1*}

Assume that $T$ satisfies Hypothesis H1. By definition $L\ge \ell>0,$ therefore we can write $L=\ell(1+\delta)$ where $\delta \ge0$. We can also observe that $L>0$ implies that $2\ell c>0$. Thus $c>0.$ Observe that $c$ is the cosine of the largest angle between two points in $T$, then $1\ge c$ since cosine is bounded above by $1$.

Copying directly from \cite{KOT} we present Hypothesis H2 as follows.

\begin{h2*}
We say that $T$ satisfies hypothesis H2 if 
\begin{enumerate}
    \item inequalities (22)-(24) in \cite{KOT} hold for $T$,
    \item for some number $\gamma '>0$ condition (28) in \cite{KOT} is satisfied
\end{enumerate}
\end{h2*}

As the title reveals, we prove that no set of points $T$ satisfies both Hypotheses H1 and H2. The inequalities I will focus on are (22) and (23), namely 

\begin{equation}\label{KOT:22}
    2\ell -L\epsilon_0>0\tag{22}
\end{equation}

and 

\begin{equation}\label{KOT:23}
    \epsilon_0>\frac{\ell \epsilon_0 +\sqrt{\ell^2\epsilon_0^2-2\ell L \epsilon_0 + L^2 \epsilon_0^2}}{2\ell -L\epsilon_0}.\tag{23}
\end{equation}

The central idea of our main proof is showing that these two inequalities contradict each other when given Hypothesis H1.

However, we first need to prove the following claim.

\begin{claim}
Let a set of points $T$ satisfy Hypothesis H1, then $\epsilon_0\ge1.$
\end{claim}

\begin{proof}
Assume to the contrary that there exists a $\ell,L$ and $c$ such that $\epsilon_0<1$ and Hyopthesis H1 is also satisfied. Recall that we can write $L=\ell(1+\delta)$ where $\delta\ge0$. Thus we can rewrite (\ref{KOT:e0}) as

\begin{align*}
    \epsilon_0 &= \frac{\ell+\sqrt{\ell^2 -2L\ell c+L^2}}{2\ell c-L}\\
    &=\frac{\ell+\sqrt{\ell^2 -2\ell^2(1+\delta) c+\ell^2(1+\delta)^2}}{2\ell c-\ell (1+\delta)}\\
    &=\frac{1+\sqrt{1 -2(1+\delta) c+(1+\delta)^2}}{2 c- (1+\delta)}.
\end{align*}

Thus we now have 

\begin{equation*}
    1>\frac{1+\sqrt{1 -2(1+\delta) c+(1+\delta)^2}}{2 c- (1+\delta)}
\end{equation*}

 Hypotheses H1 tells us that $2\ell c -L>0$. Thus $2 c- (1+\delta)$ is positive. We now have that
\begin{align*}
    2 c- (1+\delta) &> 1+\sqrt{1 -2(1+\delta) c+(1+\delta)^2} \\
    \Rightarrow 2(c-1) -\delta &>\sqrt{2(1-c)+2\delta(1-c)+\delta^2}.
\end{align*}
Since $1\ge c>0$, we have that $0\ge c-1.$ Since $\delta\ge 0$, we have that the LHS of the last inequality is non-positive and the RHS is non-negative. A contradiction since $\epsilon_0\ge 1.$
\end{proof}

Now for the main result.

\begin{theorem}
Let a set of points $T$ satisfy Hypothesis H1, then $T$ does not satisfy Hypothesis H2.
\end{theorem}

\begin{proof}
Let us rewrite $L=\ell(1+\delta)$ when $\delta\ge 0$.

Thus we can rewrite (22) as
\[2-(1+\delta)\epsilon_0>0.\]

and we can rewrite (23) as 

\[\epsilon_0>\frac{ \epsilon_0 +\sqrt{\epsilon_0^2-2(1+\delta) \epsilon_0 + (1+\delta)^2\epsilon_0^2}}{2 -(1+\delta)\epsilon_0}.\]

 Assume to the contrary that there exists a set $T$ such that H2 can be satisfied. Then there exists an $\epsilon_0$ and a $\delta$ that satisfies both inequalities (22) and (23).
If $T$ satisfies inequality (22) thus we can obtain an equivalent inequality to (23):

\[\epsilon_0(2 -(1+\delta)\epsilon_0)-\epsilon_0>\sqrt{\epsilon_0^2-2(1+\delta) \epsilon_0 + (1+\delta)^2\epsilon_0^2}.\]

With a little bit of algebra on each side, we obtain
\[-\delta\epsilon_0^2-(\epsilon_0^2-\epsilon_0)>\sqrt{(1+2\delta)(\epsilon_0^2-\epsilon_0)+\epsilon_0^2\delta^2}.\]

By the above Claim A.1, $\epsilon_0\ge1$, thus $\epsilon_0^2\ge\epsilon_0$. That combined with the fact that $\delta\ge0,$ we have that the LHS is non-positive and the RHS is non-negative. This would make the above inequality incorrect, thus we have a contradiction.

Thus when given H1, the inequality (23) is not valid when given inequality (22). Therefore the inequalities in H2 are contradictory, so H2 cannot be satisfied.
\end{proof}

\section{Discussion}

In this section, we proved existence theorems for the rotationally symmetric case, Theorem \ref{cov-rot}. Rotationally symmetric cases are not only practically useful because this case provides a model situation \cite{Oliker_1989}, but also because rotationally symmetric solutions can be used to recover nonrotationally symmetric solutions from irradiance distributions without special symmetry assumptions \cite{Oliker1987NearRS}. We also proved an existence theorem for the case where the points in the target set are sufficiently close to each other, Theorem \ref{continuous}. Theorem \ref{continuous} has the potential to lead to solutions for all kinds of exotic, Lebesgue measurable, target sets. We also proved a uniqueness theorem for the case when the target set is finite, Theorem \ref{unique}, and the general case, Theorem \ref{unique-continuous}. Thus, we make significant progress on the original formulation of the refractor problem.

For Theorem \ref{continuous}, a possible avenue for further research would be to find precise values for $\xi$ and $\gamma$ such that the theorem holds. Another potential avenue for research is to find an explicit algorithm to find proper hyperboloids for the discrete case. In addition, the author believes that Theorems \ref{continuous} and \ref{cov-rot} provide indirect evidence for the following conjecture. 

\begin{conj}
Let $T$ be a target set that satisfies Hypothesis H1 and $\epsilon_0<\lim_{t\to K^+}\frac{1}{t-1}$ for $K=\frac{\max_{x\in T} |x|}{\min_{x\in T} |x|}$ when $\epsilon_0$ is defined by (\ref{e0}).  Assume that we are given a $\gamma>0$ such that $\epsilon_0+\gamma<\lim_{t\to K^+}\frac{1}{t-1}$. Also, assume we are given a nonnegative $g\in L^1(\mathbb{S}^2)$ where $g\equiv 0$ outside $D_T^{\delta_\gamma}$ where $\delta_\gamma=\frac{1}{\epsilon_0 +\gamma}$. Assume we are given a nonnegative $f\in L^1(T)$ and $F$ is the measure defined by (\ref{lebesgue}) such that
\begin{equation}
    F(T) = \mu_{g}(D_T^{\delta_\gamma}).
\end{equation}

Then there exists an $\epsilon_M\in(\epsilon_0+\gamma,\lim_{t\to K^+}\frac{1}{t-1})$ such that there exists a convex refractor $R  \in \mathcal{R}_{convex}^{\epsilon_M} (T)$ where $R$ is a convex weak solution to the refractor problem (\ref{weak3}).
\end{conj}

\section*{Acknowledgements}

The author would like to especially thank his advisor, Prof. Vladimir Oliker, for introducing him to this problem. Also, the author would like to thank Prof. David Borthwick for helping him with the presentation of this paper.

\bibliography{nearfield}

\end{document}